\documentclass[12pt,reqno]{amsart}

\usepackage{amsmath, amssymb, latexsym, amsfonts, amsthm, url, color, float}
\usepackage[vmargin=2.2cm,hmargin=2.6cm]{geometry} 
\usepackage{graphicx}
 
\usepackage{ifthen}
\newboolean{alefonts}
\setboolean{alefonts}{true}

\ifthenelse{\boolean{alefonts}}{

\usepackage[full]{textcomp} 
\usepackage{newtxtext} 
\usepackage{cabin} 
\usepackage{zlmtt}
\usepackage[bigdelims,vvarbb]{newtxmath} 
\usepackage[cal=boondoxo]{mathalfa} 

}{}

\newtheoremstyle{sltheorems}
{10pt}
{6pt}
{\slshape}
{}
{\bfseries}
{.}
{.5em}
{\thmname{#1}\thmnumber{ #2}\thmnote{ (#3)}}

\theoremstyle{sltheorems} 
\newtheorem{Thm}{Theorem}

\newtheorem{lem}{Lemma}

\newtheoremstyle{remark}
{10pt}
{6pt}
{\rm} 
{}
{\bfseries}
{.}
{.5em}
{\thmname{#1}\thmnumber{ #2}\thmnote{ (#3)}}
 
\theoremstyle{remark} 
\newtheorem{Rem}{Remark}

\definecolor{carmine}{rgb}{0.7, 0.1, 0.09}

\makeatletter
\let\@@pmod\pmod
\DeclareRobustCommand{\pmod}{\@ifstar\@pmods\@@pmod}
\def\@pmods#1{\mkern4mu({\operator@font mod}\mkern 6mu#1)}
\makeatother

\usepackage{microtype} 

\usepackage{caption} 
\makeatletter
\patchcmd{\env@cases}{1.2}{1}{}{}
\makeatother

\usepackage[pdftex, draft=false, backref = page]{hyperref}
\hypersetup{
    colorlinks=true,
    linkcolor=blue,
    urlcolor =blue, 
    citecolor = blue}

\renewcommand*{\backrefalt}[4]{%
\ifcase #1 %
\color{red} No citations.%
\or
(p.~#2).%
\else
(pp.~#2).%
\fi
}

\usepackage{mathtools}
\begin{document}

\title[Counting ideals in abelian number fields]{Counting ideals in abelian number fields}

\author{Alessandro Languasco,  Rashi Lunia and Pieter Moree}

\date{}


\begin{abstract}
\noindent 
Already Dedekind and Weber considered the problem of counting integral ideals of norm 
at most $x$ in a given number field $K$.
Here we improve on the existing results in case $K/\mathbb Q$ is abelian and has degree
at least four.  For these fields, we obtain as a consequence 
an improvement of the available results on counting 
pairs of coprime ideals each having norm at most $x$.
\end{abstract}
 
\maketitle

\section{Introduction}
Let $K$ be a number field of degree $n_K$, 
having ring of integers ${\mathcal O}_K$ and
$$N_K(x) \coloneq \#\{\mathfrak a\in {\mathcal O}_K \colon N{\mathfrak a}\le x\},$$
denote the number of integral ideals with norm at most $x$.
Since $N_{\mathbb Q}(x)=[x]$, we may assume that $n_K\ge 2$.
It is a classic problem going back to Dedekind and Weber to approximate 
$N_K(x)$ for large $x$.
In 1896, Weber \cite{Weber} 
(see also Marcus \cite[Ch.~6]{Marcus}) proved that
\begin{equation}
\label{Weber-estim}
N_{K}(x)
= 
\rho_{K} x + 
O_{K}\big(x^{1-1/n_{K}}\big),
\end{equation}
where $\rho_K$ denotes the residue at $1$ of the 
\emph{Dedekind zeta function} $\zeta_K(s)$.

In 1912, Landau \cite[Satz~210]{EL} established that
\begin{equation}
\label{Landau-estim}
N_{K}(x)
= 
\rho_{K} x + 
O_{K}\big(x^{1-2/(n_K + 1)}\big),
\end{equation}
see also Chandrasekharan--Narasimhan \cite[p.~128]{ChandraNara1962}.
Lee \cite[Thm.~1.2]{Lee} and Lowry-Duda et al.~\cite[Thm.~3]{LTT} established a more explicit version with 
error term 
\[
O_{n_K}\bigl(|d_K|^{1/(n_K+1)}x^{1-2/(n_K + 1)}(\log z)^{n_K-1}\bigr),
\]
with $z=|d_K|$ (with $d_K$ the discriminant of $K$), 
respectively $z=x$, where the implicit constant only depends on $n_K$. Lee, for his error term, 
even made the dependence on $n_K$ explicit. These results 
improve on Ph.D. work (1971) of J.E.S. Sunley, cf.~Lee 
\cite[Thm.~1.1]{Lee}.

In 1924, Landau \cite{LOmega}
proved that
\begin{equation}
\label{Landau-Omega}
N_K(x) = \rho_{K} x + 
\Omega(x^{1/2- 1/2n_K}),
\end{equation}
which much later was slightly improved by
Szeg\H{o}--Walfisz \cite{SzegoW} and Hafner \cite{Hafner}. 
In several papers it is stated 
(without reference\footnote{
The authors challenge the reader to provide 
them with such a reference.}) that it is conjectured that
\begin{equation}
\label{conj-estim}
N_K(x) = \rho_{K} x + 
O(x^{1/2-1/2n_K+\epsilon});
\end{equation}
hence, asymptotically on $n_K$, 
the error term is of order $\sqrt{x}$.
Ram Murty and Van Order \cite{RamVanOrder} 
give a nice exposition of the early work.

For quadratic $K$, Huxley \cite{MH1, MH2} obtained 
\begin{equation}
\label{Huxley-estim}
N_{K}(x)= \rho_{K} x + O_K(x^{23/73}(\log x)^{315/146}).
\end{equation}
M\"uller \cite{Muller} showed that if $K$ is cubic, then
\begin{equation}
\label{Muller-estim}
N_{K}(x)= \rho_{K} x + O_{K, \epsilon}(x^{43/96 + \epsilon}).
\end{equation}
For number fields $K$ with $n_K \ge 3$, Nowak \cite{Now} improved the error term in the estimate of Landau and 
showed that the exponent in \eqref{Landau-estim} can be replaced by 
$1-\frac{2}{n_K}+\frac{8}{n_K(5n_K+2)}+ \epsilon $ for $3\le n_K\le 6$
and by
$1-\frac{2}{n_K}+\frac{3}{2n_K^2} + \epsilon$ for $n_K\ge 7$.
The error term was further  improved by Lao \cite{Lao} 
and  Paul--Sankaranarayanan \cite{PS})
who showed that the exponent in \eqref{Landau-estim} 
can be replaced by 
$1-\frac{3}{n_K+6} + \epsilon$ for $n_K\ge 10$.
In case $K$ is cyclotomic, it was
also shown in \cite{PS} that
exponent in \eqref{Landau-estim} 
can be replaced by 
$1-\frac{3}{n_K+5} + \epsilon$ for $n_K\ge 8$.
 
We improve here
on the results of Lao and Paul--Sankaranarayanan 
for abelian extensions $K/\mathbb{Q}$ with $n_K\ge 4$. We do so by using deep results on moments of the
Riemann zeta function and Dirichlet $L$-functions.
\begin{Thm}
\label{idealcount}
Let $K/\mathbb{Q}$ be an abelian extension with $n_K \ge 4$. Then we have
\begin{equation}
\label{idealcountwitherror}    
N_{K}(x)
= 
\rho_{K} x + 
O_{K, \epsilon}\big(x^{1-\theta_K + \epsilon}\big),
\end{equation}
with $\rho_K$ the residue of the Dedekind zeta function of $K$ at $s=1$ 
and
\begin{equation}
\label{thetak}    
\theta_K 
= \begin{cases}
\frac{4}{n_K + 4} & \textrm{for}\ \ 4 \le n_K \le 12,\\[1.5ex]
\frac{3}{n_K -\delta} 
& \textrm{for}\ \ n_K \ge 13,
\end{cases}
\end{equation}
where $0\le\delta<1$ is any real number for
which
\begin{equation}\label{subconv0}
\zeta\Bigl(\frac{1}{2} + it\Bigr) \ll_\epsilon
 \Big\vert \frac12 + it \Big\vert^{(1-\delta)/6 + \epsilon}.
\end{equation}
\end{Thm}
We note that by the 
celebrated Hardy--Littlewood--Weyl estimate we can take $\delta=0$ 
in \eqref{subconv0}; Bourgain  \cite{JB} currently holds the record here, namely $\delta = 1/14$.

For the convenience of the reader we record some values of $\theta_K$ and of the exponent in the error
term in \eqref{idealcountwitherror} occurring in Theorem \ref{idealcount} with $\delta = 1/14$
for $13\le n_K \le 15$:
\begin{table}[htp] 
\scalebox{0.8}{
\renewcommand{\arraystretch}{1.3}
\begin{tabular}{ | c | c | l | l |}
\hline
$n_K$ & $\theta_K$ in \eqref{thetak} &  \phantom{01}$1-\theta_K$ in \eqref{idealcountwitherror} & conjectural in \eqref{conj-estim}\\ \hline
$4$ & $1/2$   & $1/2  = 0.5$         & $3/8  = 0.375$       \\ \hline 
$5$ & $4/9$   & $5/9  = 0.555\dotsc$ & $2/5  = 0.4$         \\ \hline 
$6$ & $2/5$   & $3/5  = 0.6$   & $5/12 = 0.416\dotsc$ \\ \hline 
$7$ & $4/11$  & $7/11 = 0.636\dotsc$ & $3/7  = 0.428\dotsc$ \\ \hline 
$8$ & $1/3$   & $2/3  = 0.666\dotsc$ & $7/16 = 0.437\dotsc$ \\ \hline 
$9$ & $4/13$  & $9/13 = 0.692\dotsc$ & $4/9  = 0.444\dotsc$ \\ \hline 
\end{tabular}
\hskip0.25cm
\begin{tabular}{ | l | c | l | l |}
\hline
$n_K$ & $\theta_K$ in \eqref{thetak} & \phantom{012}$1-\theta_K$ in \eqref{idealcountwitherror} & conjectural in \eqref{conj-estim}\\ \hline
$10$ & $2/7$    & $5/7     = 0.714\dotsc$ & $9/20  = 0.45$        \\ \hline  
$11$ & $4/15$   & $11/15   = 0.733\dotsc$ & $5/11  = 0.454\dotsc$ \\ \hline  
$12$ & $1/4$    & $3/4     = 0.75$        & $11/24 = 0.458\dotsc$ \\ \hline  
$13$ & $42/181$ & $139/181 = 0.767\dotsc$ & $6/13  = 0.461\dotsc$ \\ \hline  
$14$ & $14/65$  & $51/65   = 0.784\dotsc$ & $13/28 = 0.464\dotsc$ \\ \hline  
$15$ & $42/209$ & $167/209 = 0.799\dotsc$ & $7/15  = 0.466\dotsc$ \\ \hline  
\end{tabular}
}
\end{table}

For $n_K=4$, the error term in Theorem \ref{idealcount} is $O(x^{1/2+\epsilon})$.
In Remark \ref{RMT-remark} we sketch a conjectural approach that allows one to obtain an 
error term of the size  $x^{1/2+\epsilon}$ for $n_k\ge 5$ too.
 
The same technique used in the proof of Theorem \ref{idealcount} 
can also handle the cases $n_K=2,3$ with an error term 
$\ll_{K,\epsilon} x^{1/2+\epsilon}$,
which is unfortunately weaker than
the results of Huxley and M\"uller mentioned before.
Their technique relates the ideal 
counting result to a generalized divisor problem
but, currently, it is not 
clear how to generalize Huxley--M\"uller's work to higher degrees cases.

The proof of Theorem \ref{idealcount} heavily relies on results
about moments of the Dirichlet $L$-functions attached to Dirichlet
characters modulo $q$. 
It is currently not known how  to obtain an estimate as sharp as the above one
for non-abelian number fields, as their Dedekind zeta function is not factorable 
as a finite product of only Dirichlet $L$-functions.

\subsection{Counting coprime ideals}
 
In many textbooks on elementary analytic number theory,  for example 
Apostol \cite[Thm.~3.9]{Apostol}, one can find a count of the number of coprime positive integers $a,b$ with $a,b\le x$ due
to Mertens (1874). 
A variation of his proof, using 
Weber's estimate \eqref{Weber-estim}, then leads to
the following generalization, see Sittinger \cite{Sittinger}, 
valid as $x$ tends to infinity,
\begin{equation}
\label{Sittingerror}    
	\sum_{\substack{N\mathfrak a,N\mathfrak b\le x
			\\ \mathfrak a+\mathfrak b=\mathcal O_{K}}} 1
	=
	\frac{N_K(x)^2 }{\zeta_K(2)}
	+O_{K}(x^{2-1/n_K}\log x).
\end{equation}
The same argument, but now
rather based on \eqref{Landau-Omega} and \eqref{conj-estim}, suggests that the optimal exponent in the error term in \eqref{Sittingerror}
is $\frac{3}{2}-\frac{1}{2n_K}$. 
The current best general result has the exponent
$ 2-\frac{3}{n_K + 6}+ \epsilon$ in the
error term for $n_K \ge 10$ (see the works of Takeda and Takeda--Koyama \cite{WTpreprint, TK}).
Here we will improve on this in case $K/\mathbb Q$ is abelian of degree  $\ge 4$.
 \begin{Thm}
\label{coprimeidealcount}
Let $K/\mathbb{Q}$ be an abelian extension with $n_K \ge 4$ and $\theta_K$ 
be as in \eqref{thetak}.
Then,
\begin{equation}
\label{coprimeidealcountwitherror}
	\sum_{\substack{N\mathfrak a,N\mathfrak b\le x
			\\ \mathfrak a+\mathfrak b=\mathcal O_{K}}} 1
	=
	\frac{N_{K}(x)^2 }{\zeta_{K}(2)}
	+
	O_{K, \epsilon}(x^{2-\theta_K+ \epsilon}),
\end{equation}
\end{Thm}

Mertens \cite{FM} established \eqref{coprimeidealcountwitherror} with error term $O(x\log x)$ for $K= \mathbb{Q}$. For quadratic extensions,
it holds with error term 
$O_K(x^{96/73}(\log x)^{315/146})$ by the Huxley's estimate \eqref{Huxley-estim}. 
 If $K$ is cubic, M\"uller's estimate \eqref{Muller-estim} gives rise to an error term of size 
$O_{K, \epsilon}(x^{139/96 + \epsilon})$ in \eqref{coprimeidealcountwitherror}.
Takeda \cite{Takeda}, assuming the Extended Lindel\"of Hypothesis (ELH) holds, 
proved Theorem \ref{coprimeidealcount} with the exponent $2-\theta_K$ replaced by $3/2$ 
for all algebraic number fields
$K$; we remark that our 
unconditional result leads to the same estimate as Takeda's in case $n_K=4$.
For $n_K\ge 5$ another conjectural approach capable of obtaining 
the exponent $3/2+\epsilon$ in 
Theorem \ref{coprimeidealcount}
follows from the argument in Remark \ref{RMT-remark}.

\section{Proof of Theorem \ref{idealcount}}
Our proof of Theorem \ref{idealcount} for a number field $K$ will crucially rely on properties of its \emph{Dedekind zeta function} 
$\zeta_K(s)$, which for $\Re (s)>1$ is defined by
$$
\zeta_K(s)=\sum_{\mathfrak{a}} \frac{1}{(N{\mathfrak{a}})^{s}}.
$$
Here, $\mathfrak{a}$ runs over 
the non-zero ideals in  the ring of integers ${\mathcal O}_K$ of $K$ and $N{\mathfrak{a}}$ denotes the norm
of $\mathfrak{a}$. 
It is known that $\zeta_K(s)$ can be analytically continued to $\mathbb C \setminus \{1\}$,
and that at $s=1$ it has a simple pole with residue $\rho_K$. For an abelian extension $K/\mathbb{Q}$, 
we  have the factorization (see, e.g., Lang \cite{Lang})
\begin{equation}
\label{product}    
\zeta_K(s)= \prod_{j=0}^{n_K-1}L(s, \chi_j),
\end{equation}
where $\chi_0$ is the principal character and $\chi_j$ is a primitive Dirichlet character of conductor $q_j$.
By the conductor discriminant formula, the conductors $q_j$ divide the discriminant $d_K$ of $K$. For other elementary properties
of $\zeta_K(s)$ we refer to the book of Marcus \cite[Ch.~7]{Marcus}.

Our proof of Theorem \ref{idealcount} rests on three lemmas, of which
the first holds for general number fields $K$. It is
a form of Perron's formula with error term and can be obtained as
a consequence of a modification 
of Landau's lemma \cite[Lemma~11]{PS} (see also Chandrasekharan's book \cite[p.~115]{KC}).

\begin{lem}\label{ModLandau}
Let $K$ be a number field. For $1< \eta < 2$, $T>0$ and $x>2$, 
\begin{equation*}
    \Big\vert\sum_{n \le x}a_n -\frac{1}{2\pi i}\int_{\eta - iT}^{\eta + iT}\zeta_K(s)\frac{x^s}{s}\ ds\Big\vert
    < c_K(\epsilon) 
    \Bigl( \frac{x^{\eta}}{T(1- \eta)} +
\frac{x^{1 + \epsilon}}{T} + x^{\epsilon} \Bigr),
\end{equation*}
where the constant $c_K(\epsilon)$ depends only on $K$ and $\epsilon$ and is independent of $\eta, T$ and $x$.
\end{lem}
For the second, we will use a refinement of the Phragm\'en--Lindel\"of theorem due to Rademacher \cite[Thm.~2]{Rademacher}
to bound the Dedekind zeta function in the right half of the critical strip.
\begin{lem}\label{Rademacher}
Let $K$ be an abelian number field of degree $n_K$
and $\epsilon > 0$ be fixed. 
For $1/2 \le  \sigma \le 1+ \epsilon$ and $t\in \mathbb R$, we have
\begin{equation*}
\vert \zeta_{K}(\sigma + it)| 
\ll_{K , \epsilon}
(1 + |t|)^{(n_K - \delta)(1-\sigma) / 3 + \epsilon},
\end{equation*}
where $0\le\delta<1$ 
is any real number for which 
the estimate \eqref{subconv0} holds.
\end{lem}
\begin{proof}

We start by bounding $\zeta_K(s)$ on the vertical lines $\sigma=1/2$ and $\sigma=1+\epsilon$.
Using \eqref{product},
the subconvexity bound \eqref{subconv0}
 and 
\begin{equation}\label{subconv}
L\Bigl(\frac{1}{2} + it, \chi \Bigr) 
\ll_{K, \epsilon} 
\Big\vert \frac12 + it \Big\vert^{1/6 + \epsilon},
\end{equation}
(see, e.g., Prachar \cite[p.~328, eq.~(4.12)]{Prachar})
we deduce that
$$
\Bigl\vert \zeta_{K}\Bigl(\frac{1}{2} + it\Bigr) \Bigr\vert \le  B_{K, \epsilon} 
\Big\vert \frac12 + it \Big\vert^{(n_K - \delta)/6 + \epsilon}
$$
for some constant $B_{K, \epsilon}$.
Currently, the best known value of $\delta$ is $1/14$ (Bourgain \cite{JB}).
For $\sigma=1+$ $\epsilon$, 
we use the trivial bound  
$$
\vert \zeta_{K}(1+ \epsilon + it) \vert \le \zeta_K( 1+ \epsilon)
$$
for all $t \in \mathbb{R}$.
Now applying Rademacher's result with
$f(s) = \zeta_K(s)$,
$a = 1/2$, $b = 1+ \epsilon$, 
$\alpha = (n_K- \delta)/6 + \epsilon$,
$\beta = 0$ and $Q =0$, we obtain
\[
\vert \zeta_K(\sigma + it) \vert
\le
C_{K, \epsilon }
|\sigma + it|^{((n_K-\delta)/3 + 2\epsilon)( 1+ \epsilon -  \sigma)/(1 + 2\epsilon) } 
\ll_{K , \epsilon}
(1 + |t|)^{(n_K-\delta)(1-\sigma) / 3 + \epsilon}.
\qedhere
\popQED
\] 
\end{proof}

The third gives an upper bound on the first moment of the Dedekind zeta function of 
an abelian number field, which might be of independent interest.
\begin{lem}\label{moment}
Let $K$ be an abelian number field with $n_K = [K \colon {\mathbb Q}] \ge 4$ and $\epsilon > 0$ be fixed. 
For $T \ge 2$ we have 
\begin{equation*}
    \int_T^{2T}\Bigl\vert \zeta_{K}\Bigl(\frac{1}{2} + it\Bigr) \Big\vert\, dt
    \ll_{K, \epsilon} 
    \begin{cases}
    T^{n_K/8 + 1/2   + \epsilon} & \ \textrm{for}\ \ 4 \le n_K \le 12,\\
    T^{(n_K - \delta)/6 + \epsilon} &\ \textrm{for}\ \  n_K \ge 13,
    \end{cases}
\end{equation*}
where $0\le\delta<1$ 
is any real number for which 
the estimate \eqref{subconv0} holds.
\end{lem}
\begin{proof}
We will 
estimate the first moment of $\zeta_K$ 
by exploiting its factorization into Dirichlet
$L$-functions as given by \eqref{product} and
using the moment 
estimates \eqref{moments1}-\eqref{moments3}.
We shall make use of the currently available bounds on moments of Dirichlet $L$-functions.
For the fourth moment 
using the results of Ingham \cite{Ing}
and 
Montgomery  \cite[Thm.~10.1]{MontgomeryTopics}, we have
\begin{equation}
\label{moments1}
\int_{T}^{2T}\ \Bigl\vert
\zeta\Bigl(\frac{1}{2} + it \Bigr)
\Bigr\vert^4 dt
\ll
T^{1 + \epsilon}
\qquad
\textrm{and}
 \qquad 
\int_{T}^{2T}\ \Bigl\vert
L\Bigl(\frac{1}{2} + it, \chi \Bigr)
\Bigr\vert^4 dt
\ll
qT^{1 + \epsilon},
\end{equation}
where $q$ is the conductor of $\chi$.
From the works of Heath-Brown \cite{HB1} and Meurman \cite[p.~25, Thm.~4]{Meurman} we have
\begin{equation}
\label{moments2}
\int_{T}^{2T}\ \Bigl\vert
\zeta\Bigl(\frac{1}{2} + it \Bigr)
\Bigr\vert^{12} dt
\ll
T^{2+ \epsilon}
\qquad
\textrm{and}
 \qquad 
\int_{T}^{2T}\ \Bigl\vert
L\Bigl(\frac{1}{2} + it, \chi \Bigr)
\Bigr\vert^{12} dt
\ll q^3T^{2+ \epsilon}.
\end{equation}
Further, we use the 
following bound on the sixth moment of the Riemann zeta function (see Ivi\'c \cite[Thm.~8.3]{Ivic}):
\begin{equation}
\label{moments3}
\int_{T}^{2T}\ \Bigl\vert
\zeta\Bigl(\frac{1}{2} + it \Bigr)
\Bigr\vert^6\, dt
\ll
T^{5/4 + \epsilon}.
\end{equation}

First we consider the case  when $4 \le n_K \le 12$ is even. Let $m_K = 6-\frac{n_K}{2}$. 
We use \eqref{moments1} for 
$m_K$ Dirichlet $L$-functions and 
\eqref{moments2} for the remaining $n_K- m_K$ functions.
Recalling \eqref{product}, this  gives
\begin{align}
\notag
	\int_{T}^{2T}\ \Bigl\vert \zeta_{K}\Bigl(\frac{1}{2} + it\Bigr) \Big\vert\ dt 
& \le 
\Big(\int_{T}^{2T}\
	\Bigl\vert \zeta\Bigl(\frac{1}{2} + it\Bigr)
	\Bigr\vert^{4} dt\Big)^{1/4}
\prod_{j= 1}^{m_K -1}\Big(\int_{T}^{2T}\
	\Bigl\vert L\Bigl(\frac{1}{2} + it, \chi_j \Bigr)
	\Bigr\vert^{4} dt\Big)^{1/4}
	\\\notag
& \hskip2cm \times	\prod_{j= m_K }^{n_K - 1}\Big(\int_{T}^{2T}\
	\Bigl\vert L\Bigl(\frac{1}{2} + it, \chi_j \Bigr)
	\Bigr\vert^{12} dt\Big)^{1/12}
	\\ \label{first-case}
&  \ll_{K, \epsilon}  T^{m_K/4 + (n_K - m_K)/6+\epsilon} 
= T^{n_K/8 + 1/2 +\epsilon} . 
 \end{align}
 
Now let $4 \le n_K \le 12$ be odd and $m_K = 6-\frac{n_K + 1}{2} $. We use \eqref{moments3} for 
the Riemann zeta function, \eqref{moments1} for $m_K$
Dirichlet $L$-functions and  \eqref{moments2} 
for the remaining $n_K- (m_K + 1)$ functions.
Recalling \eqref{product} and applying H\"older's inequality, we have
\begin{align}
\notag
	\int_{T}^{2T}\ \Bigl\vert \zeta_{K}\Bigl(\frac{1}{2} + it\Bigr) \Big\vert\ dt 
 &\le    \Big( \int_{T}^{2T}\ \Bigl\vert \zeta\Bigl(\frac{1}{2} + it\Bigr) \Big\vert^{6} dt \Big)^{1/6}
\prod_{j= 1}^{m_K }\Big(\int_{T}^{2T}\
	\Bigl\vert L\Bigl(\frac{1}{2} + it, \chi_j \Bigr)
	\Bigr\vert^{4} dt\Big)^{1/4}
	\\ \notag
& \hskip2cm \times \prod_{j= m_K +1 }^{n_K - 1}\Big(\int_{T}^{2T}\
	\Bigl\vert L\Bigl(\frac{1}{2} + it, \chi_j \Bigr)
	\Bigr\vert^{12} dt\Big)^{1/12}
	\\
	\label{second-case}
 &\ll_{K, \epsilon}    T^{m_K/4+ (n_K - m_K)/6 + 1/24 +\epsilon} 
= T^{n_K/8 + 1/2 +\epsilon} .
\end{align}
Finally, when $n_K \ge 13$, we use \eqref{moments2} 
for twelve Dirichlet $L$-functions and 
the subconvexity bounds \eqref{subconv} for the remaining ones.
Recalling \eqref{product} and using H\"older's inequality we have
\begin{align}
\notag
	\int_{T}^{2T}\ \Bigl\vert \zeta_{K}\Bigl(\frac{1}{2} + it\Bigr) \Big\vert\ dt 
  &\ll_{K, \epsilon}  
 T^{(n_K - 12 - \delta)/6 + \epsilon}
 \prod_{j= n_K -12 }^{n_K - 1 }\Big(\int_{T}^{2T}\
	\Bigl\vert L\Bigl(\frac{1}{2} + it, \chi_j \Bigr)
	\Bigr\vert^{12} dt\Big)^{1/12}
	\\
	\label{third-case}
 &\ll_{K, \epsilon}  
 T^{(n_K - \delta)/6 + \epsilon}. 
 \end{align}
Collecting \eqref{first-case}-\eqref{third-case},
Lemma \ref{moment} follows.
\end{proof}

\begin{Rem}
By solving an integral linear programming problem,
it can be shown that the estimates 
in Lemma \ref{moment} are
optimal (given the
information in  \eqref{moments1}-\eqref{moments3} and the other available ones
about the integral moments of $\zeta(s)$ on the critical line).
For $4\le n_K\le 12$, the goal is to minimize the exponent of $T$ in the final estimate; 
such an exponent, using the H\"older inequality, the informations about the 
moments  (up to twelfth one) and the subconvexity estimates
mentioned before, can be written as
$f(A_j) = \frac{1}{6}\bigl(n_K-\sum_{j=2}^{12} A_j\bigr) +  \sum_{j=2}^{4}A_j/j 
+ \sum_{j=5}^{12} A_j (j+4)/(8j)  $, where $A_j$ denotes
the number of functions for which the $j$-th moment is used.
One also has to consider the conditions coming from the H\"older inequality,
which translates as $\sum_{j=2}^{12} A_j/j = 1$,
the fact that we are using a set of $n_K$ functions
($\sum_{j=2}^{12} A_j\le n_K$) and that some moments
are available for the Riemann zeta function only
($A_j \in \{0,1\}$  for every $j\in\{5,\dotsc, 11\}$).
For every $4\le n_K\le 12$ the Simplex Algorithm applied to this 
problem gave  the solutions described in the proof of Lemma  \ref{moment}.
\end{Rem}
   
\begin{proof}[Proof of Theorem \ref{idealcount}]
Let $\epsilon>0$. 
Using  Lemma \ref{ModLandau}, we obtain
$$
N_K(x) = \frac{1}{2\pi i}
\int_{1+\epsilon -iT}^{1+ \epsilon +iT}\zeta_{K}(s)\frac{x^s}{s}\, ds
+ O_{K, \epsilon}\Big(\frac{x^{1+ \epsilon}}{T}\Big),
$$
where $1 < T \le x$ is a parameter to be chosen later. On moving the 
line of integration to $\Re(s) = 1/2$ and using Cauchy's theorem we obtain
\begin{equation}
\label{Cauchy2}
N_{K}(x)
= 
\rho_{K} x 
+ \frac{1}{2\pi i}
\Big( \int_{1+\epsilon -iT}^{1/2 - iT} 
+ \int_{1/2 -iT}^{1/2 + iT}
+ \int_{1/2 + iT}^{1+\epsilon + iT}\Big)
\zeta_{K}(s)\frac{x^s}{s} \, ds
+ O_{K, \epsilon}\Big(\frac{x^{1+ \epsilon}}{T}\Big).
\end{equation}
By  Lemma \ref{Rademacher}, for $1/2 \le \sigma \le 1+ \epsilon$ and $|t| \ge 1$, we have
\begin{equation*}
 \zeta_{K}(\sigma + it)  \ll_{K,\epsilon} 
 |t|^{(n_K - \delta)(1-\sigma) / 3+ \epsilon}.
\end{equation*}

Now we can estimate the contribution on the horizontal
segments in \eqref{Cauchy2} as follows: 
\begin{align}\label{error3}
	\frac{1}{2\pi i	}
	\Big( \int_{1+\epsilon -iT}^{1/2 - iT} 
	+ \int_{1/2 + iT}^{1+\epsilon + iT}\Big)
	\zeta_{K}(s)\frac{x^s}{s}\, ds
	&\ll 
	\frac{1}{T}
	\int_{1/2}^{1+ \epsilon}
	\vert \zeta_{K}(\sigma + iT) \vert\, x^{\sigma}\, d\sigma\nonumber\\
	& \ll_{K, \epsilon} 
	\max_{1/2 \le \sigma \le 1+ \epsilon}
	x^{\sigma} T^{	(n_K - \delta)(1-\sigma) / 3 -1+\epsilon}
	\nonumber\\
	& \ll_{K, \epsilon} 
	x^{1/2+\epsilon}T^{-1}\bigl( x^{1/2}+ T^{(n_K-\delta)/6}\bigr).
\end{align} 
Next we estimate the contribution of the vertical segment to \eqref{Cauchy2}.
A dyadic dissection gives that
\begin{align*}
	\frac{1}{2\pi i	}
	\int_{1/2 -iT}^{1/2 + iT}
	\zeta_{K}(s)\frac{x^s}{s}\, ds
	& \ll_K 
	x^{1/2} + x^{1/2}\int_{2}^T
	\frac{\vert \zeta_{K}(1/2 + it) \vert}{t}\, dt \nonumber\\
	& \ll_K 
	x^{1/2} + x^{1/2}\log T \Bigl(\max_{4\le T_1\le T}\frac{1}{T_1}\int_{T_1/2}^{T_1}\ 
	\Bigl\vert \zeta_{K}\Bigl(\frac{1}{2} + it\Bigr) \Big\vert\, dt\Bigr)
\end{align*}
and hence, using Lemma \ref{moment} we have 
\begin{equation}\label{error4}
	\frac{1}{2\pi i}
	\int_{1/2 -iT}^{1/2 + iT}
	\zeta_{K}(s)\frac{x^s}{s}\ ds
	 \ll_{K, \epsilon} 
   \begin{cases}
    x^{1/2+\epsilon}\,T^{n_K/8 - 1/2} &\ \textrm{for}\ \ 4 \le n_K \le 12,\\[1ex]
    x^{1/2+\epsilon}\,T^{(n_K-\delta)/6 - 1} &\ \textrm{for}\ \ n_K \ge 13.
    \end{cases} 
\end{equation}	  

Using \eqref{Cauchy2}, \eqref{error3} and \eqref{error4}, and recalling
\eqref{thetak}, we finally have
\begin{equation*}
N_{K}(x)
= 
\rho_{K} x + 
O_{K, \epsilon}
\bigl(
x^{1/2 + \epsilon}T^{-1}
\bigl( x^{1/2} + T^{1/2\theta_K}
\bigr) 
\bigr). 
\end{equation*}
Theorem \ref{idealcount} follows on  choosing $T= x^{\theta_K}$. 
\end{proof}

\begin{Rem}
\label{RMT-remark}
The proof of Theorem \ref{idealcount} is based on suitable estimates for moments
of the Riemann zeta function and 
of the Dirichlet $L$-functions, and it is clear that more detailed information about 
their higher moments, or
about the moments of the Dedekind zeta function, would allow one to obtain sharper estimates 
of the error term for $n_K\ge 5$.

Based on random matrix theory heuristics, a weaker form of the conjecture in \cite{WH} implies that
\begin{equation}
\label{Heap-conjecture}
\int_{0}^{T}
\Bigl\vert \zeta_K\Bigl(\frac{1}{2} + it\Bigr) \Bigr\vert^{2m}\, dt
\ll_{K, \epsilon} T^{1+ \epsilon}
\end{equation}
for any positive real number $m$. 
Applying this for $m=1/2$ leads to the estimate $x^{1/2}T^{\epsilon}$ in \eqref{error4}
and hence the error term in Theorem \ref{idealcount} would be 
$\ll_{K,\epsilon} x^{1/2+\epsilon}$ for $n_K\ge 5$ too.
In \cite[Thm.~ 1.3]{MilinovichTB2014} it is shown that 
GRH implies (a slightly refined form of) \eqref{Heap-conjecture}
in the case in which $K$ is a finite solvable Galois extension of $\mathbb Q$.
The estimate in \eqref{Heap-conjecture} also follows from the Extended
Lindel\"of Hypothesis; Takeda exploited this in his paper \cite{Takeda},
see the proof of Theorem 3.2 there.
\end{Rem}

\begin{proof}[Proof of Theorem \ref{coprimeidealcount}]
Recall the definition of 
the M\"obius function over ideals:
\[
\mu(\mathfrak a) = 
\begin{cases}
1 & \textrm{if}\  \mathfrak a={\mathcal O}_K, \\
0 & \textrm{if}\ \mathfrak a \subseteq {\mathfrak p}^2 \ \textrm{for some prime ideal}\  {\mathfrak p}, \\
(-1)^r & \textrm{if}\ \mathfrak a ={\mathfrak p}_1 {\mathfrak p}_2 \dotsm {\mathfrak p}_r \ \textrm{for distinct prime ideals}\ 
{\mathfrak p}_1, {\mathfrak p}_2, \dotsc, {\mathfrak p}_r. \\
\end{cases}
\]
Using an inclusion-exclusion argument, we deduce that
$$
\sum_{\substack{1 \le N\mathfrak a,N\mathfrak b\le x
		\\ \mathfrak a+\mathfrak b=\mathcal O_{K}}} 1
 = 
\sum_{\substack{0 \neq {\mathfrak a}\subseteq \mathcal O_{K}}}\mu(\mathfrak a)N_{K}\Big(\frac{x}{N\mathfrak a}\Big)^2
 =
\sum_{1 \le  N\mathfrak{a} \le x }
\mu(\mathfrak a)\Bigl(\rho_{K}^2 \frac{x^2} {(N\mathfrak a)^2}
+ O_{K,\epsilon}\Big(\Big(\frac{x} {N\mathfrak a}\Big)^{2-\theta_K+ \epsilon}\Big)\Bigr),
$$
where $N_K(y)=0$ for $y<1$, so that
 the sum into the middle term is over non-zero integral ideals with norm at most $x$,
 and in the third sum we squared out the result in Theorem \ref{idealcount}.
 Since the sum in the error term is bounded by a convergent series, we can simplify this and obtain
\begin{equation}\label{eqn1}
\sum_{\substack{1 \le N\mathfrak a,N\mathfrak b\le x
		\\ \mathfrak a+\mathfrak b=\mathcal O_{K}}} 1
 =
 \rho_{K}^2 x^2
\sum_{1 \le  N\mathfrak{a} \le x }
\frac{\mu(\mathfrak a)} {(N\mathfrak a)^2}
+ O_{K, \epsilon}\big(x^{2-\theta_K + \epsilon}\big).
\end{equation}
Recalling that $\zeta_{K}(s)^{-1} = \sum_{\mathfrak{a}} \mu(\mathfrak a)(N{\mathfrak{a}})^{-s}$ for 
$\Re(s)>1$, we also get
$$
 \sum_{1 \le  N\mathfrak{a} \le x }
 \frac{\mu(\mathfrak a)}{(N\mathfrak a)^2}
 =
 \frac{1}{\zeta_{K}(2)}
-
\sum_{N\mathfrak{a} > x }
\frac{\mu(\mathfrak a) }{(N\mathfrak a)^2}.
$$
On noting that $N_K(m) - N_K(m-1) \le d(m)^{n_K} \ll_{n_K,\epsilon} m^{\epsilon}$, we see that
$$
\sum_{N\mathfrak a >x} \frac{1}{(N\mathfrak a)^2} 
\ll \int_x^{\infty}\frac{y^{\epsilon}}{y^2}\ dy 
\ll_{n_K,\epsilon} x^{-1 + \epsilon},
$$
and hence
$$
 \sum_{1 \le  N\mathfrak{a} \le x }
 \frac{\mu(\mathfrak a)}{(N\mathfrak a)^2}
 =
 \frac{1}{\zeta_{K}(2)}
 + 
O_{n_K, \epsilon}(x^{-1 + \epsilon}).
$$
Inserting this estimate into \eqref{eqn1} we have
$$
\sum_{\substack{1 \le N\mathfrak a,N\mathfrak b\le x
		\\ \mathfrak a+\mathfrak b=\mathcal O_{K}}} 1
 =
 \frac{\rho_K^2 x^2}{\zeta_{K}(2)}
+
O_{K, \epsilon}(x^{2-\theta_K + \epsilon}).
$$
The proof is now concluded on
comparing this with the estimate 
for $N_K(x)$ given in Theorem \ref{idealcount}.
\end{proof}

\begin{Rem}
For an integer $m\ge 3$,
a variation of the above proof, in combination with the identity
$$
\sum_{\substack{1 \le N\mathfrak{a}_1, \ldots, N\mathfrak{a}_m\le x
	\\ \mathfrak{a}_1+ \dotsm+ \mathfrak{a}_m=\mathcal O_{K}}} 1
 = 
\sum_{\substack{0 \neq {\mathfrak a}\subseteq \mathcal O_{K}}}\mu(\mathfrak a)N_{K}\Big(\frac{x}{N\mathfrak a}\Big)^m,
$$
yields, with $\theta_K$ as in \eqref{thetak},
\begin{equation*}
\sum_{\substack{1 \le N\mathfrak{a}_1, \ldots, N\mathfrak{a}_m\le x
	\\ \mathfrak{a}_1+ \dotsm + \mathfrak{a}_m=\mathcal O_{K}}} 1
 = 
\frac{N_{K}(x)^m }{\zeta_{K}(m)}
	+
	O_{K, m, \epsilon}(x^{m-\theta_K+ \epsilon}),
\end{equation*}
which improves on results of Sittinger--DeMoss \cite{SittingerDeMoss}, Takeda \cite{WTpreprint} and Takeda--Koyama
\cite{TK}. 
\end{Rem}

\medskip
\noindent \textbf{Acknowledgment}. 
This article 
arose as a side problem in our attempts in
answering a question of Hendrik 
Lenstra and Elie Studnia, see \cite{FLLM}.
We are grateful that they contacted the authors regarding this problem. We thank
Valentin Blomer, Olivier Bordell\`es, Steve Fan, Sanoli Gun,
 Winston Heap, Ikuya Kaneko and Djordje Milicevic 
 for helpful feedback. We also thank Anu Hirvonen (MPIM-librarian) for enthusiastic and efficient 
assistance.
 
The paper was written during a postdoctoral visit of the second author to the
Max-Planck-Institut f\"ur Mathematik. She is grateful to the third author for his 
mentorship and guidance. She is also grateful to MPIM for this opportunity.

\vskip1cm
\begin{flushleft} 
Alessandro Languasco, Universit\`a di Padova,\\
Dipartimento di Ingegneria Informatica - DEI, \\
via Gradenigo 6/b, 35131 Padova, Italy.\\
e-mail: alessandro.languasco@unipd.it
\end{flushleft}

\medskip
\begin{flushleft}
Rashi Lunia, Max Planck Institute for Mathematics\\ 
Vivatsgasse 7, 53111, Bonn, 
Germany.\\  
e-mail: lunia@mpim-bonn.mpg.de 
\end{flushleft}

\medskip
\begin{flushleft}
Pieter Moree, Max Planck Institute for Mathematics\\ 
Vivatsgasse 7, 53111, Bonn, 
Germany.\\  
e-mail: moree@mpim-bonn.mpg.de 
\end{flushleft}
\end{document}